\documentclass[10pt,oneside]{amsart}
\usepackage{amssymb, amscd, amsmath, amsthm, latexsym, hyperref, amsrefs}
\usepackage{pb-diagram, fancyhdr, graphicx, psfrag, enumerate}
\usepackage[text={6.3in,8.5in},centering,letterpaper,dvips]{geometry}
\newcommand{\compactlist}{\begin{list}{$\bullet$}{\setlength{\leftmargin}{1em}}}
\def\fr{{\sf{fr}}}

\def\F{{\mathbb F}}

\def\cc{{\mathbb C}}
\def\pp{{\mathbb P}}

\def\call{\mathcal{L}}

\def\calk{\mathcal{K}}
\def\cals{\mathcal{S}}

\def\calk{\mathcal{K}}

\def\cs{\mathbin{\#}}

\newcommand{\spinc}{\ifmmode{{\mathfrak s}}\else{${\mathfrak s}$\ }\fi}
\newcommand{\spinct}{\ifmmode{{\mathfrak t}}\else{${\mathfrak t}$\ }\fi}

\newcommand{\fig}[2] { \includegraphics[scale=#1]{#2} }

\newtheorem{theorem}{Theorem}[section]

\theoremstyle{definition}
\newtheorem{definition}[theorem]{Definition}

\begin{document}
\title{Null-homologous  unknottings}
\author{Charles Livingston}
\thanks{This work was supported by a grant from the National Science Foundation, NSF-DMS-1505586.}
\address{Charles Livingston: Department of Mathematics, Indiana University, Bloomington, IN 47405 }
\email{livingst@indiana.edu}

%%%%%%%ABSTRACT%%%%%%%%%%%%%%

\begin{abstract}   Every knot can be unknotted with two generalized twists; this was first proved by Ohyama.  Here we prove that any knot of genus $g$ can be unknotted with $2g$ null-homologous twists and that there exist genus $g$ knots that cannot be unknotted with fewer than $2g$ null-homologous twists.  

\end{abstract}

\maketitle

%%%%%%%Section%%%%%%%%%%%%%%

\section{Introduction}

Figure~\ref{fig7} illustrates an  operation that can be performed on a knot,  twisting a set  of parallel strands. In this example, orientations are shown and the linking number of the twist is one in absolute value.  In  1994, Ohyama~\cite{MR1297516} proved the unexpected result that  every knot can be unknotted with two twists.  Here we will give an alternative perspective on the proof of this theorem  and use this to show that   every knot of three-genus $g = g(K)$ can be unknotted using $2g$ null-homologous twists, that is, with all twists having linking number 0.  It will also be shown that $2g(K)$ is the best possible bound.

\begin{figure}[h]
\fig{.5}{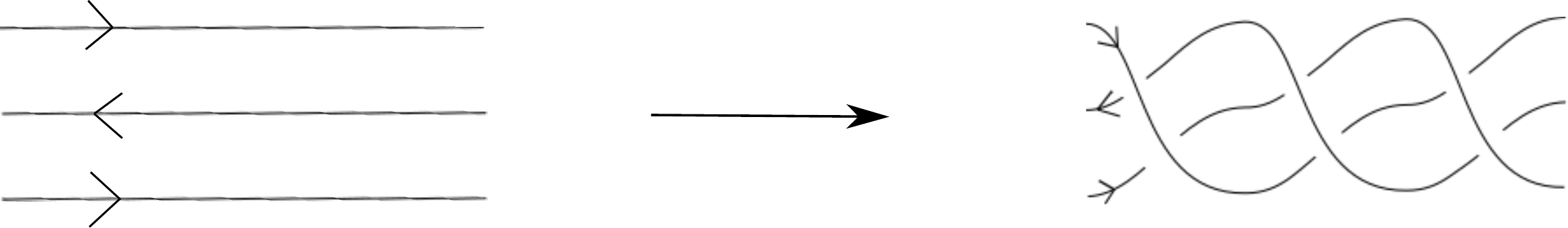}
\caption{A twist operation on a portion of a knot.}
\label{fig7}
\end{figure}
\vskip.05in

There is extensive literature concerning the general problem of unknotting.  Some of this research concerns the question of which knots can be unknotted with a single operation that  consists of performing many full twists with no constraint on the linking number; a small sample of   papers on this topic includes~\cite{MR1177414, MR2233011, MR1282760, MR1194998}.   More generally, one can consider an equivalence relation on knots generated by twisting, with constraints on the linking number and number of twists; the starting points of research on this topic is a paper by Fox~\cite{MR0131871}, and a few later papers  include~\cite{MR3544442,MR881755,MR1049833}.  

This work is closely related to the study of  classes of unknotting operations; a good initial reference is~\cite{MR1075165}.  Our approach is based on simplifying a Seifert surface for the knot; one can also consider the problem of simplifying closed surfaces and compact three-manifolds embedded  in $S^3$ using twisting operations.  Research along these lines includes~\cite{MR3676048, MR3807598}, in which similar techniques to those used here appear.   Geometric operations on a knot correspond to algebraic operations on its Seifert matrix, and this in turn leads to the notion of algebraic unknotting.  See, for instance,~\cite{MR1043226, MR1727885, 2017arXiv170910269B}.

Our focus  on linking number 0 arose because of its relationship to problems in four-dimensional aspects of knot theory: if a knot $K$ can be unknotted with $k$ positive and $j$ negative null-homologous full twists, then $K$ bounds a null-homologous disk embedded in the punctured connected sum $\#_k \cc \pp^2 \#_j \overline{\cc \pp}^2$.  This is a topic that has received considerable attention; a few early references include~\cite{MR603768, MR0246309, MR1195083, MR0248851,MR852974, MR561245}. From the four-dimensional perspective, this in turn leads to the question of converting a knot into a knot with Alexander polynomial 1 (which would then be topologically slice~\cite{MR679066}); this is explored, for instance, in~\cite{2017arXiv170910269B}.

\vskip.05in

\noindent{\it Acknowledgement}\ \  In~\cite{MR1936979},  the main result of this paper was stated without proof.    Recently,  Maciej Borodzik and Lukas Lewark, who needed  the result,   noted that a follow-up paper promised in~\cite{MR1936979}  never appeared.  Maciej and Lukas also had outlined their own proof of the desired result.  Here we present the argument.   Thanks are also due to Makoto Ozawa and Akira Yasuhara for valuable comments and for identifying important references.

\vskip.05in

\section{Surgery descriptions of knots and unknotting}\label{sec:surgery}

 Throughout this paper, we will use surgery descriptions of   knots, links, and their cyclic covering spaces.  A basic reference  is the text by Rolfsen~\cite[Chapter 9H]{MR0515288}.   More details can be found in~\cite{MR1707327} and original sources such as~\cite{MR0467753}.

An $n$--component {\it framed}  link $L$ in a three-manifold $M$ is the isotopy class of an oriented embedded copy of the disjoint union of  $n$ copies of $S^1$  in $M$ along with the choice of an  isotopy class of a nowhere vanishing section of the normal bundle to $L$. For oriented links in $S^3$, linking numbers define a natural correspondence between framings on each component and integers.

In the introduction we   described the general twisting operation.  We now make this formal in a way that facilitates the proof of the unknotting theorem.

A surgery diagram for a  framed $n$--component  link $L$ in a three-manifold $M$ consists of an oriented link  $L_\cals = \{\call_1, \ldots , \call_n, \cals_1, \ldots, \cals_k\}$  in $S^3$   for which each component  has  an  integer framing; it should have the property that $\cals = \{ \cals_1, \ldots, \cals_k\}$ is  a surgery description of $M$ and  $\call = \{ \call_1, \ldots, \call_n\}$   represents  the framed link $L$.   The framing of $\call_i$ is denoted $\fr(\call_i)$.

For our work, the key result concerning modifications of surgery diagrams states that two such diagrams represent isotopic links in $M$ if they are related by a sequence of three types of moves along with the inverse of the second move:  (1)  diagram isotopy; (2) removing   an unknotted component $\cals_j$ with framing $\epsilon = \pm 1$ and adding a full twist to the remaining strands of $L_\cals$ that pass through $\cals_j$, twisting left or right depending on whether $\epsilon = 1$ or $\epsilon = -1$, respectively; (3) {\it sliding} a component of $L_\cals$ over an  $\cals_i$ (other than itself), adjusting framings appropriately.  If $\epsilon =1$, the second move, {\it blowing down} $\cals_j$, decreases   the framing of each remaining component by the square of the linking number; if  $\epsilon = -1$, then it increases framings by the square of the linking number.   Details are presented in the references;  in the work below we will summarize  how framings can be tracked using associated linking matrices.

\begin{definition} A link $L \subset S^3$ can be {\it unlinked} with $k$ twists if there is a surgery diagram of $L$,   in which  both $\{\call_1, \ldots , \call_n\}$ and $\{\cals_1, \ldots , \cals_k \}$ are unlinks.  The unlinking is called {\it null-homologous} if each $\cals_i$ is null-homologous in the complement of $\call$.
\end{definition}
Note that the choice of framing for $\call$ is not relevant in this definition.  Also note  that in a surgery diagram for $S^3$ for which   $\cals$ is an unlink, all framings are necessarily $\pm 1$. Thus, this definition of unlinking corresponds  to the  definition that involves applying twists to a standard diagram of $L$.

%%%%%%%Section%%%%%%%%%%%%%%

\section{The Light Bulb Trick}

A well-known result, called {\it The Light Bulb Trick}, states that any two (unoriented)  knots in $S^1 \times S^2$, each of which meets a nonseparating two-sphere in exactly one point, are isotopic.  There is a simple generalization for connected sums of $S^1 \times S^2$  which quickly yields the following result.  To set up notation for the proof, in the statement of the theorem we make precise the notion that a subset of   $\cals$ forms a set of meridians for some of the $\call_i$.

\begin{theorem}\label{thm:unlink} Let  $L \subset S^3$ be  a link with surgery diagram $ \{\call_1, \ldots , \call_n, \cals_1, \ldots, \cals_k\}$.   Suppose that for some $m \le \min\{n,k\}$, the link $\{\cals_1, \ldots , \cals_m\}$ has all framings 0 and bounds a set of disjoint disks $D_i$ in $S^3$ such that $D_i \cap \left( \call_1\cup \ldots \cup \call_n \right)$ consists of exactly one point, and that   intersection point is  on $\call_i$.  Then $L$ has a surgery diagram  $$L =  \{\call_1', \ldots , \call_n', \cals_1, \ldots, \cals_k\} = \{ \call', \cals\}$$ for which   $\{\call_1', \ldots , \call_m'\}$ is an unlink bounding a set of disjoint disks with interiors in the  complement of   $\call'$.  Furthermore, it can be arranged that if $f_i$ is an arbitrary set of integers satisfying $f_i = \fr( \call_i) \mod 2$, then $\fr(\call_i') = f_i$
\end{theorem}

\begin{proof}

The isotopy is constructed by repeatedly sliding  elements of $\{\call_i\}_{i\le n}$ over elements of $\{\cals_i\}_{i\le m}$.
On the left in Figure~\ref{fig:unknotslide}, a schematic diagram of a portion of the diagram for $L$ is presented.  In this portion of the diagram, it is possible that $i = j$.  In the full diagram, it is also possible that some of the $\cals_i$ for $i>m$ intersect the disks $D_i$ nontrivially.  What is essential is that full intersection of  $\call$ with $D_i$ consist of only the one point, $\call_i \cap D_i$.  

\begin{figure}[h]
 \fig{.4}{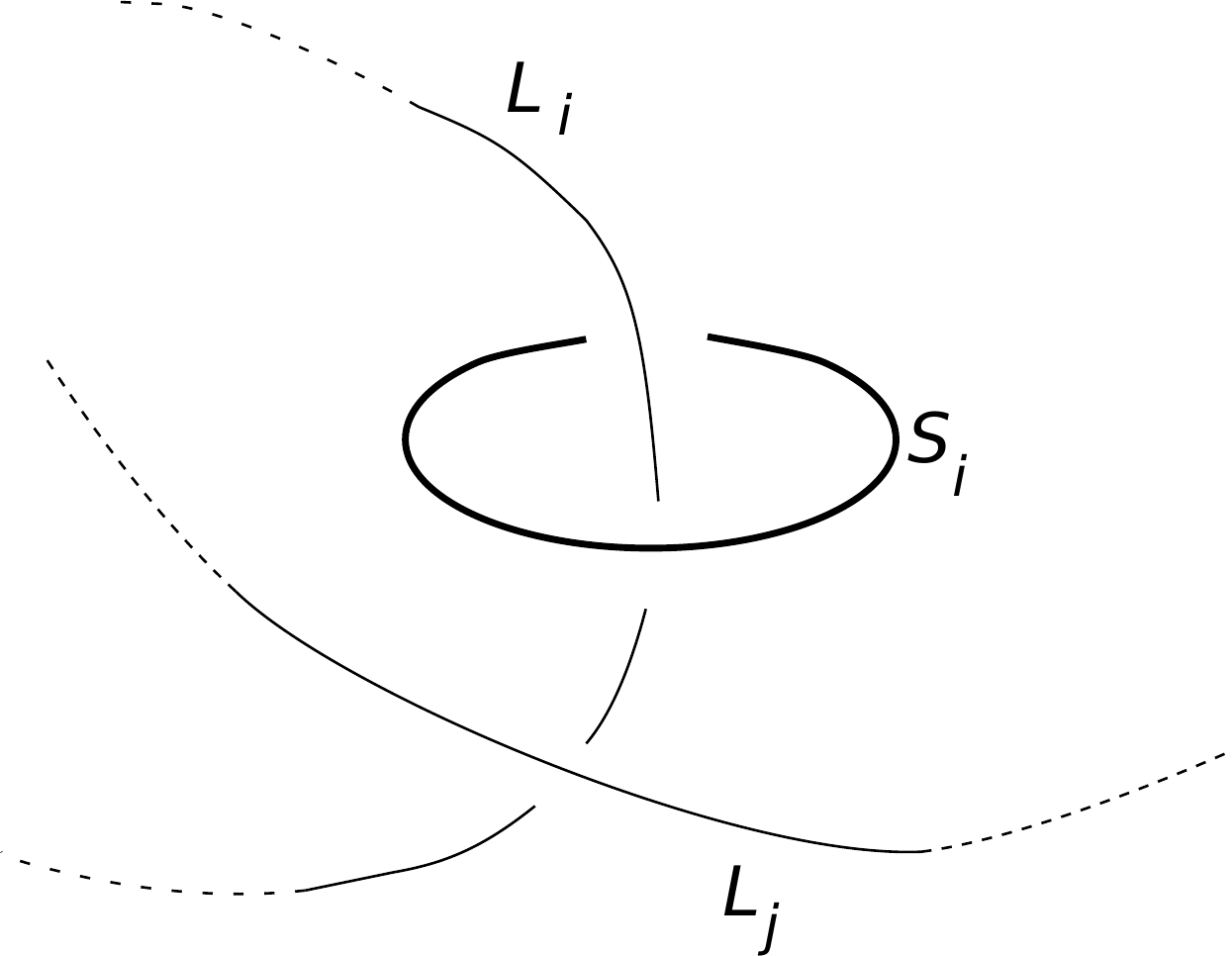}  \hskip.75in \fig{.4}{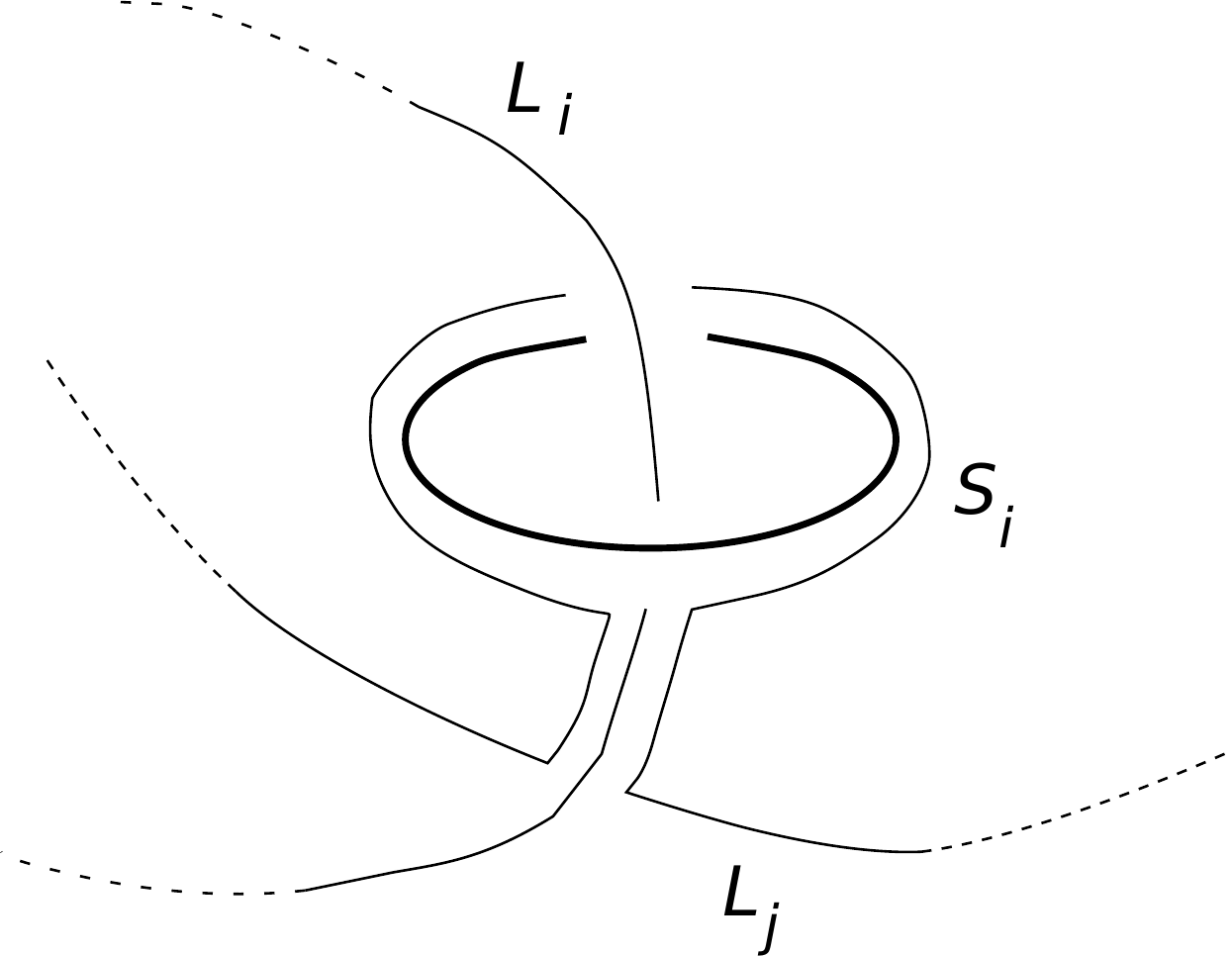} 
\caption{Sliding over a surgery curve to realize a crossing change.}
\label{fig:unknotslide}
\end{figure}

The schematic  on the right in  Figure~\ref{fig:unknotslide} illustrates the effect of sliding $\call_j$ over $\cals_i$; recall that $\cals_i$ has framing 0.  The   effect is to change one crossing between $\call_i$ and $\call_j$.  If $i = j$, then the framing of $\call_i$ is changed by $\pm 2$. (It is also possible that the linking of $\call_j$ with some $\cals_\alpha$ for $\alpha > m$ is changed, but the resulting link diagram continues to satisfy the initial conditions of the theorem.) It follows that arbitrary   crossing changes between components of   $\{\call_i\}_{i\le n}$ and components of $\{\call_i\}_{i\le m}$ can be performed so that in the new link, which we denote   $\{\call'_i\}_{i\le n}$, we have that  $\{\call'_i\}_{i\le m}$  is an unlink  separated from $\{\call'_i\}_{i > m}$.  Note that since no crossing changes between components of $\{\call_i\}_{i > m}$ were performed, $\{\call'_i\}_{i > m}  =  \{\call_i\}_{i > m}$. 

Once the link   $\{\call_i'\}_{i\le n}$ is  constructed, for $i \le m$ a Reidemeister move I can be preformed to $\call_i'$ to put a small left-handed kink in the diagram.  That crossing can be changed to be right-handed by sliding $\call_i'$ over $\cals_i$.  This  does not change the link type of     $\{\call_i'\}_{i\le n}$,  but it changes the framing of $\call_i'$ by 2.  Thus, this move and its inverse can change all framings by arbitrary multiples of 2.

\end{proof}

%%%%%%%Section%%%%%%%%%%%%%%

\section{Unknotting knots with a single pair of twists}

We begin by considering unknotting a single knot, offering a new perspective on   Ohyama's theorem. Our approach keeps track of framings.

\begin{theorem}[Unknotting Theorem] \label{thm:main1} Any knot $K$ in $S^3$ can be unknotted with two twists of opposite sign.  If $K$ has framing $f$, then in the diagram $\{\calk, \cals_1, \cals_2\}$ in which $\calk$ is unknotted, it can be arranged   that $\calk$ has framing $f'$ for any integer $f'$  satisfying $f + f' \equiv 1 \mod 2$. The absolute values of the linking numbers of the two unknotting curves with $K$ differ by 1.
\end{theorem}

\begin{proof} 
Figure~\ref{fig1} is a schematic   framed    link diagram of a knot $K$ in $S^3$.  Sliding the $-1$ framed component over the $+1$ framed component, and then sliding $\calk$ over the $+1$ framed component yields a diagram as shown on the right in Figure~\ref{fig1}.

\begin{figure}[h]
\hskip-3.3in\fig{.8}{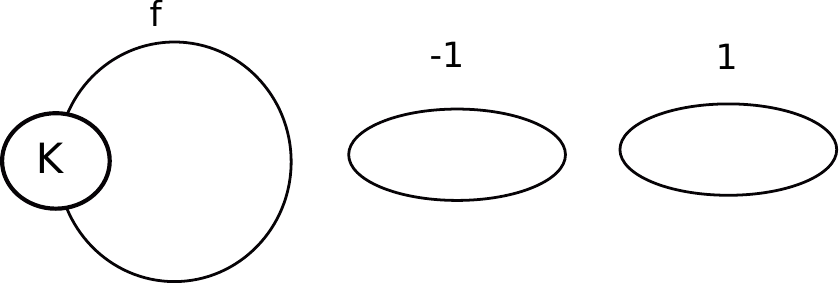}\\
\vskip-.55in\hskip.7in$\xrightarrow{\hspace*{.5in}}$\\
\vskip-.55in\hskip4in \fig{.8}{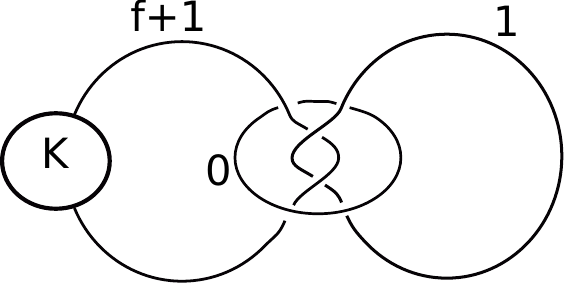}\\
\caption{Modifying a surgery diagram of $K$}
\label{fig1}
\end{figure}

The framing/linking matrix is given by 
$$
\begin{pmatrix}% or pmatrix or bmatrix or Bmatrix or ...
f +1& 1 &1 \\
1 & 0 & 1 \\
1 & 1 &1 \\
\end{pmatrix}
$$

If we denote this surgery diagram for $K$ by  $\{\calk , \cals_1, \cals_2\}$, with $\cals_1$ denoting the 0--framed surgery component, then it satisfies the conditions of Theorem~\ref{thm:unlink}, and thus we can unknot $\calk$.  If the number of slides of $\calk$ over $\cals_1$ is (algebraically) $\alpha$, the linking matrix for the resulting link   $\{\calk', \cals_1, \cals_2\}$ is given by

$$
\begin{pmatrix} % or pmatrix or bmatrix or Bmatrix or ...
f +2\alpha+1& 1 &\alpha+1 \\
1 & 0 &1 \\
\alpha+1 & 1 & 1 \\
\end{pmatrix}.
$$\vskip.05in
 
 Sliding the $0$--framed curve over the $1$--framed surgery curve converts $\{\cals_1, \cals_2\}$ into an unlink with framings $-1$ and $-1$.  The resulting framing/linking matrix becomes
 
$$
\begin{pmatrix} % or pmatrix or bmatrix or Bmatrix or ...
f +2\alpha+1& -\alpha &\alpha+1 \\
-\alpha & -1 &0\\
\alpha+1 & 0 & 1 \\
\end{pmatrix}.
$$\vskip.05in

 As in the proof of Theorem~\ref{thm:unlink}, by placing kinks in $\calk$ and then sliding it over $\cals_1$, the choice of $\alpha$ is seen to be arbitrary.

\end{proof}

%%%%%%%Section%%%%%%%%%%%%%%

\section{Null-homologous unknotting}

\begin{theorem}\label{thm:main2} If $g(K) = g$, then $K$ can be unknotted with $2g$ null-homologous twists.
\end{theorem}

\begin{proof} %%%%%%%%%%%%%%%%%%%%%%%%%%%%%%%%%%%%%%%%BEGIN

Let $F$ be a Seifert surface for $K$ built from a disk by adding $g$ pairs of bands.  Then $F$ can be  illustrated as in Figure~\ref{fig2}.  Each of the bands is drawn with  a small gap in it, indicating that the bands are perhaps knotted, linked together, and twisted.  There is a  symplectic basis of $H_1(F)$ represented by simple closed curves  built from the cores of those bands.  We denote these curves by  $\{a_1, b_1, \ldots , a_g, b_g\}$, as shown. 

\begin{figure}[h]  \fig{.8}{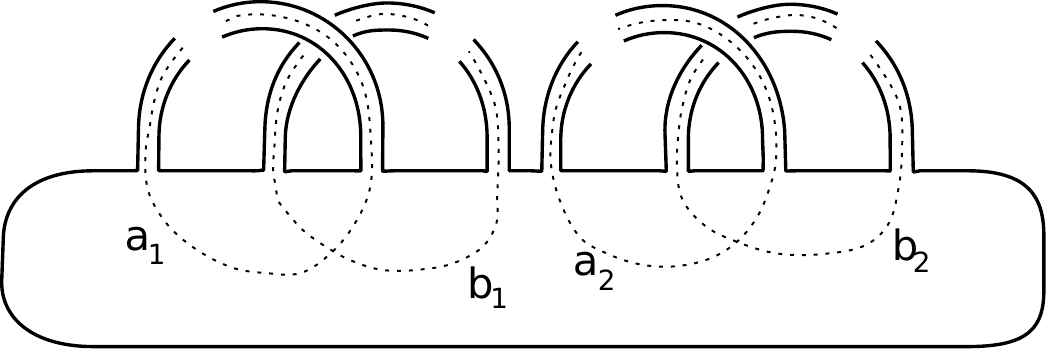}\\
\caption{Basic surface}
\label{fig2}
\end{figure}

 Let $\bar{b}_i$ denote the negative push-offs of the $b_i$ from $F$.  Thus, in Figure~\ref{fig2}   the visible portion of $\overline{b_i}$ lies under $a_i$.  If the link $\{a_1, \ldots , a_g\}$ forms an unlink with all Seifert framings 0,   as shown in Figure~\ref{fig5},   and that unlink lies everywhere over  $\{ \bar{b}_1 , \ldots , \bar{b}_g\}$, then $K$ is an unknot: the Seifert surface can be ambiently surgered to form a disk.   Figure~\ref{fig5} is drawn to shown the $a_i$ forming a trivial link along which $F$ can be surgered.  The $b_i$ bands are still drawn broken, to indicate that they are perhaps knotted, twisted, and link, but it is assumed they pass everywhere under the $a_i$ bands.
 
 \begin{figure}[h]  \fig{.8}{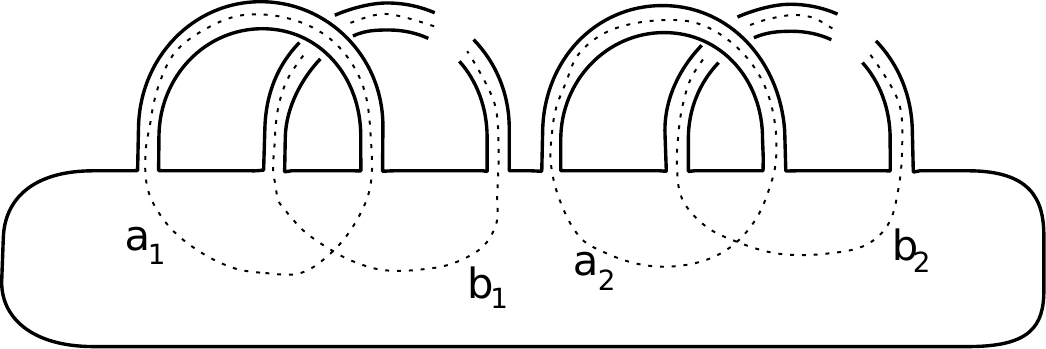}\\
\caption{Basic surface}
\label{fig5}
\end{figure}

Thus, to unknot  $K$, we can proceed as follows.  Introduce $g$ pairs of two component unlinks  in the complement of $F$ and perform $+1$ and $-1$ surgery on each.  This is illustrated in Figure~\ref{fig3}.  As in the proof of the unknotting theorem, Theorem~\ref{thm:main1}, we can slide each $-1$ curve and an $a_i$ curve (and its corresponding band) over a $+1$ surgery curve to arrange that each $a_i$ band has a small linking circle on which 0--surgery is performed.  We are now in the setting of Theorem~\ref{thm:unlink}, with the set  $\{a_i\}_{i\le g}$ corresponding to the set  $\{\call_i\}_{i\le m}$.  However, instead of sliding the $a_i$, we work with the corresponding $a_i$ bands. 

As in the earlier arguments, a sequence of slides (of the $a_i$ bands) over the  $0$--framed linking curves can ensure that: (1) each $a_i$ curve is unknotted; (2) the set of $a_i$ curves forms an unlink, and finally; (3) by sliding $b_i$ curves over the $0$--framed surgery curves, that  the link formed from the  $\overline{b}_i$ curves is split from the link formed from the $a_i$ curves, lying completely beneath it.

\begin{figure}[h]  \fig{.8}{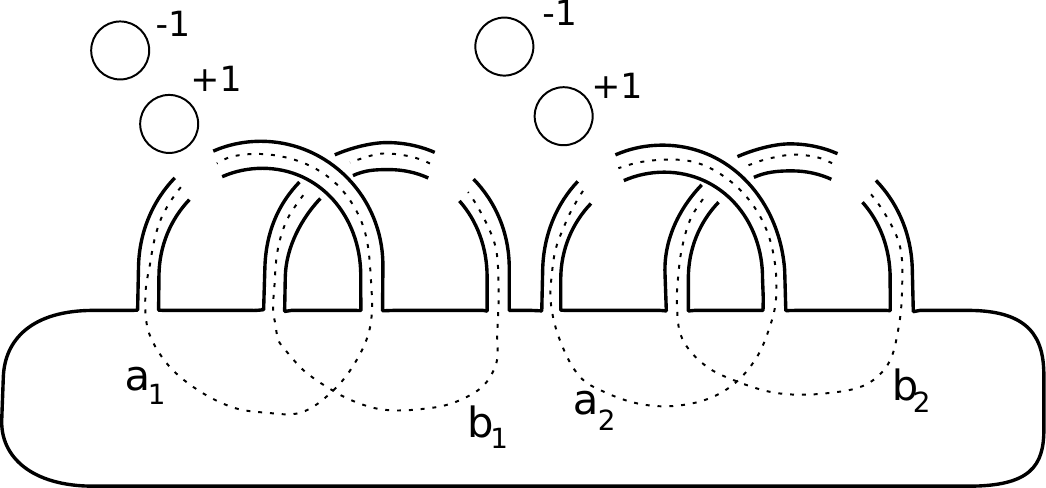}\\
\caption{Basic surface}
\label{fig3}
\end{figure}

It remains to arrange that the $a_i$ curves all have framing 0.   In Theorem~\ref{thm:main1}, we saw that the framings of the $a_i$ can be changed by any odd integer.  Thus, if the initial Seifert pairing of each $a_i$ curve were odd, then we would be done.  Consider each pair $\{a_i, b_i\}$.  If either has framing odd,     perhaps after switching the labels, we would be able to proceed.   If both are even, we can consider a new symplectic basis $\{a_i + b_i, b_i\}$.  Letting $V_{\calk}$ denote the Seifert form, we have

\begin{equation*}
\begin{split}
V_\calk(a_i +b_i , a_i +b_i) & =  V_\calk(a_i,a_i) +V_\calk(a_i,b_i) +V_\calk(b_i a_i) + V_\calk(b_i,b_i)  \\
& \equiv V_\calk(a_i,b_i) -V_\calk(b_i , a_i) \mod 2 \\
&\equiv a_i \cap b_i  \equiv 1 \mod 2.   
\end{split}
\end{equation*}

By using a symplectic basis in which $a_i + b_i$ is one of the basis elements, we have arranged that it has odd framing.  Thus, the argument is complete.

\end{proof} %%%%%%%%%%%%%%%%%%%%%%%%%%%%%%%%%%%%%%%%END

%%%%%%%Section%%%%%%%%%%%%%%

\section{$2g$  Converse}

\begin{theorem}\label{thm:0linking}  For every $g >0$, there exists a genus $g$ knot that cannot be unknotted with fewer than $2g$ null-homologous twists. 
\end{theorem}

\begin{proof} %%%%%%%%%%%%%%%%%%%%%%%%%%%%%%%%%%%%%%%%BEGIN

If $K$ can be unknotted with $k$ null-homologous twists, the first  homology of the infinite cyclic cover of $S^3 - K$, $H_1(X_\infty(K), {\mathbb F})$, has rank at most $k$ as an ${\mathbb F}[t, t^{-1}]$--module for any field $\mathbb F$.  (See~\cite[Chapter 7C]{MR0515288}.)  In general, this homology group is presented by $V - tV^{\sf T}$, where $V$ is a Seifert matrix for $K$.  If $K$ is a genus one knot with Seifert matrix 
$$ \begin{pmatrix}  
0 & 1 \\
2 & 0 \\
\end{pmatrix},
$$
then letting $\mathbb F_3$ be the field with three elements,  
$$H_1(X_\infty(K), {\mathbb F_3}) \cong \left( \frac{\mathbb F_3[t, t^{-1}]}
{\left<1- 2t\right>} \right) \oplus\left( \frac{\mathbb F_3[t, t^{-1}]} {\left<2 - t\right>} \right) \cong  \left( \frac{\mathbb F_3[t, t^{-1}]} {\left<t+1\right>} \right)^2 $$ as an $\F_3[t, t^{-1}]$--module.  This is of rank 2.

The knot $gK$ is now seen to have  $H_1(X_\infty(gK), {\mathbb F_3})$ of rank $2g$ as an  ${\mathbb F_3}[t, t^{-1}]$--module, and thus it provides the necessary example to conclude the proof.
\end{proof}

\section{Concluding remarks}
There is an interesting point of overlap between our null-homologous unknotting result and Ohyama's original theorem.  That is in the case of $g(K) = 1$.  For such a knot, Ohyama says it can be unknotted with a positive and negative twist   with linking  numbers  $(k, k \pm 1)$ for {\it some} integer $k$.  Our result says that it can be unknotted with such twists of linking numbers $(0,0)$.  It is not difficult to modify a proof of Ohyama's theorem to show that for every  knot $K$ and for {\it every} value of $k$, there is a $(k,k \pm 1)$ unknotting of $K$.   

One question is to determine for each knot $K$, for which pairs $(m,n)$  there is an unknotting with these linking numbers.  Notice that for any integers $m$ and $n$, the connected sum of torus knots $T(|m|,  m \pm 1) \cs T(|n|, n\pm 1)$ can be unknotted with a pair of twists of linking numbers $(|m|,|n|)$, with the signs of the twists determined by the signs of $m$ and $n$.   On the other hand, for any given knot, there are potential  obstructions to unknotting with linking numbers any given pair $(m,n)$.  If $g(K) =1$, it seems possible that the condition is $|k - m| \le 1$.  These issues will be the subject of further research.

Such questions can be considered from a four-dimensional perspective.  According to~\cite{MR0246309, MR0248851}, every knot bounds a smoothly embedded disk in  $ X =   \cc \pp^2 \#  \overline{\cc \pp}^2 \setminus B^4$, and by Ohyama's result, this can be improved:  for every integer $k$, any knot $K$ bounds a disk in   $ X$ that represents $(k, k \pm 1) \in H_2(  X , \partial X)$.  If $g(K) = 1$, then it also bounds a disk representing $(0, 0) \in H_2(  X , \partial X)$.   Again, what other possibilities can occur? 
%%%%%%%Section%%%%%%%%%%%%%%

%%%%%%%%%%%%%%%%%%%%%%%%%%%%%%%%%%%%%%%%END

%%%%%%%BIBLIOGRAPHY%%%%%%%%%%%%%%

\bibliography{../../BibTexComplete.bib}
\bibliographystyle{plain}

\end{document}